\documentclass{amsart}

\usepackage{amsmath,amssymb,amsthm,amsfonts, mathrsfs}
\usepackage[all]{xy}
\usepackage{enumerate}
\usepackage{graphicx}

\newtheorem{theorem}{Theorem}[section]

\newtheorem{problem}[theorem]{Problem}

\newtheorem{conjecture}[theorem]{Conjecture}
\newtheorem{remark}[theorem]{Remark}

\newtheorem{proposition}[theorem]{Proposition}

\newtheorem{definition}[theorem]{Definition}

\begin{document}

\title{Extension properties of orbit spaces for proper  actions  revisited}

\author{Sergey A. Antonyan\\[3pt]}

\address{Departamento de  Matem\'aticas,
Facultad de Ciencias, Universidad Nacional Aut\'onoma de M\'exico,
 04510 M\'exico City, M\'exico.}
\email{antonyan@unam.mx}

\begin{abstract} Let  $G$ be a locally compact Hausdorff group. We   study orbit spaces   of equivariant absolute neighborhood extensors ($G$-${\rm ANE}$'s) for the class
 of all proper $G$-spaces  that are metrizable by a $G$-invariant metric. 
 We  prove that if a  $G$-space $X$ is a  $G$-${\rm ANE}$ and all  $G $-orbits in $X$ are metrizable, then  the $G$-orbit space  $X/G$ is an {\rm ANE}. If $G$ is either a Lie group or an almost connected  group, then  for any  closed  normal subgroup $H$ of $G$,   the $H$-orbit space  $X/H$ is a $G/H$-{\rm ANE} provided that all $H$-orbits in $X$ are metrizable.
\end{abstract}

\thanks{The author was supported  by grants  IN-100123 from PAPIIT (UNAM) and  A1-S-7897 from CONACYT}

\maketitle
\markboth{SERGEY A. ANTONYAN}{ORBIT SPACES  OF $G$-${\rm ANE}$'S}

\section {Introduction}\label{intro}

In  \cite[Question 4]{ant:aspects}, 
the following problem about extension properties of orbit spaces  was  posed.

\begin{problem} [Orbit Space Problem]\label{P:0} Let $G$ be a compact group and $X$ a $G$-{\rm ANR}. Is then the orbit space $X/G$ an  {\rm ANR}?
\end{problem}

This problem  has been  solved affirmatively 
first in \cite[Theorem 8]{ant:88} for $G$ a compact metrizable group.
   The result was extended  in    \cite[Theorem 1.1]{ant:99}  to the case of any
 compact Hausdorff group $G$ and any $X\in G$-A(N)E with all orbits $G(x)\subset X$ metrizable. More precisely,  the following more general result was proved in \cite[Theorem 1.1]{ant:99}. 
  
\begin{theorem}[\cite{ant:99}]\label{T:compactcase}
Let $G$ be a compact group and  $H$  a closed normal subgroup of $G$. Suppose $X$ is a $G$-space such that all  $H$-orbits in $X$ are metrizable. 
 If  $X$ is a $G$-${\rm ANE}$ (respectively, a $G$-${\rm AE}$),  then the $H$-orbit space $X/H$ is a $G/H$-${\rm ANE}$ (respectively, a $G/H$-${\rm AE}$). 
\end{theorem}
  
  This result has been widely applied to the study of the topology of  Banach-Mazur compacta  (see \cite{ant:98}, \cite{ant:00}, \cite{ant:03}).  Other applications  can be found in \cite{ant:88}, \cite{bas:94} and \cite{zhur:90}.
  
\medskip

The possibility of a further extension of Theorem \ref{T:compactcase}
 to the case of proper actions (in the sense of R. Palais \cite{pal:61}) of locally compact groups has been studied in several papers by  the author \cite{ant:99}, \cite{ant:jap}, \cite{ant:proper}, \cite{ant:fund09}.

\medskip

The main purpose of this article is to carefully review the proofs of various subsequent generalizations of Theorem \ref{T:compactcase}  and make  corrections and improvements where necessary. Along the  way we will provide very short proofs for  the following three theorems about preservation of the  extension properties by the orbit space functor.   More general versions of these theorems are also presented in Section \ref{lifting}.

\begin{theorem}[Proper actions of locally compact groups and $H=G$]\label{T:0}
Let $G$ be a locally compact  group and $X$   a proper $G$-space such that  all  $G$-orbits in $X$ are metrizable. 
 If  $X$ is a $G$-${\rm ANE}$ then the $G$-orbit space  $X/G$ is an {\rm ANE}. 
\end{theorem}

\smallskip

      A proof  of this theorem was provided in  \cite[Theorem 6.4]{ant:jap}. In that proof         
      the following affirmation, that we state here in the form of a proposition, was used.

 \begin{proposition}\label{retractdirect}
 Let $G$ be a topological group and $K$ any closed subgroup of  $G$. If $S$ is a $K$-space, then the $G$-orbit space
 $({G \times_K S})/G$ is homeomorphic to a retract of  the $K$-orbit space $({G \times_K S})/K$.
 \end{proposition}
 
   Here ${G \times_K S}$ denotes the twisted product of $G$ and $S$ with respect to the subgroup $K$ defined in Section \ref{def}.
      
The argument for the proof of this statement  given in the proof of \cite[Theorem 6.4]{ant:jap}, unfortunately,  works correctly only for an Abelian acting group $G$.  Namely,  in  that  proof 
 the formula  $(G \times_K S)/G\cong G/K\times S/K$ was used  which, however, is valid only for an Abelian group $G$ (see \cite[Proposition 2]{ant:99}).
 
  Although  we will need  Proposition \ref{retractdirect} here only in the case of a compact subgroup $K\subset G$, we provide
 a detailed proof  even for any closed subgroup $K\subset G$ in the Appendix.
  Thus,  the gap in the proof  of \cite[Theorem 6.4]{ant:jap} is easily filled.
 
 \medskip

   The second theorem is the following.
 
   \begin{theorem}[Proper actions of almost connected groups]\label{T:almost} Let $G$ be an almost connected locally compact group,  $H$  a closed normal subgroup of $G$, and $X$   a proper $G$-space that admits a $G$-invariant metric.   If 
 $X$ is a $G$-${\rm ANE}$, then the $H$-orbit space  $X/H$ is a $G/H$-${\rm ANE}$. 
\end{theorem}   
  The proof of this theorem given in \cite[Theorem 3(1)]{ant:99} is correct only for compact subgroups $H\subset G$. For an arbitrary closed subgroup $H\subset G$ our argument in  \cite[Theorem 3(3)]{ant:99}  is based on \cite[Proposition 8(1)]{ant:99} the proof of which, unfortunately, contains a gap.
  The proof of  Theorem \ref{T:almost} given in Section  \ref{almost}  is based on the following well-known proposition.

          \medskip
  
    \begin{proposition}\label{abels} Let $G$ be an almost connected locally compact group, $K$ a compact  maximal  subgroup of $G$, and $S$  a $K$-space. Then $S$ is a $K$-equivariant retract of the twisted product $G\times_K S$.
    \end{proposition}        
    \begin{proof}
     This result  follows from a result of  Abels \cite[Theorem~2.1]{abe:74} according to which $G\times_K
S$ is $K$-homeomorphic to a product $T\times S$ endowed with the diagonal action of $K$, where $T$
is a finite-dimensional linear $K$-space. In this case the map $(t, s)\mapsto (0, s)$ is a
$K$-equivariant retraction of $T\times S$ onto $\{0\}\times S$ which, in turn, is $K$-homeomorphic
to $S$.
 \end{proof}

  \medskip
  
 The third theorem is the following.
 
   \begin{theorem}[Proper actions of  any Lie groups]\label{T:00} Let $G$ be a Lie  group,  $H$  a closed normal subgroup of $G$, and $X$   a proper $G$-space. If 
 $X$ is a $G$-${\rm ANE}$, then the $H$-orbit space  $X/H$ is a $G/H$-${\rm ANE}$. 

\end{theorem} 

 In \cite[Theorem 1.1]{ant:proper}      a proof  of this theorem was given even for  any locally compact acting group $G$. Again, this proof used a formula  (see \cite[formula (3.3)]{ant:proper}), which is  only valid for Abelian groups. Below we give  a very brief proof of this theorem in the case of proper actions of arbitrary Lie groups, which is practically the most important case.
 This proof is based on the following  result we  proved 
 in \cite[Proposition 4.1]{ant:proper}.
  
  \begin{proposition}[\cite{ant:proper}]\label{P:1} Let $G$ be a Lie  group, $K$ a compact   subgroup of $G$, and $S$  a $K$-space. Then $S$ is a neighborhood $K$-equivariant retract of the twisted product $G\times_K S$.
\end{proposition}    
  
    \medskip

    Regarding the \lq\lq $G$-AE version\rq\rq \  of the above  results, we have the following theorem proven in \cite[Theorem 7.1]{ant:fund09}.

    \begin{theorem}\label{T:}
Let $G$ be a locally compact group and $X$ any   $G$-${\rm
{\rm AE}}$. Assume that $H$ is an almost connected normal subgroup of $G$
such that all $H$-orbits in $X$ are metrizable. Then the $H$-orbit space $X/H$ is a $G/H$-${\rm
AE}$.
\end{theorem}

\smallskip

 We note that almost connectedness  of $H$ is  essential in this theorem. Indeed, let    $G=\Bbb R$,
 the reals,  $X=\Bbb R$  and $H=\Bbb Z$,
 the integers.  Then the translation action is a proper action of $G$ on $X$, and  by
  \cite[Theorem 4.4]{abe:78}, \ $X$ is   a $G$-${\rm {\rm AE}}$.
  However $X/H$,  being a circle, is not  an ${\rm {\rm AE}}$.
  
        \medskip
        
        In Section \ref{lifting}   we strengthen  
 Theorems~\ref{T:0}, \ref{T:almost} and \ref{T:00},  discarding  in their statements the hypothesis about the properness of the $G$-space $X$. This is achieved by using the lifting properties of equivariant embeddings.

        \medskip
        
Before passing to the details of the proofs, it is convenient to recall some auxiliary notions and results.
 
\medskip

\section {Some basic definitions and auxiliary results}\label{def}

Throughout the paper the letter $G$ will denote a locally compact Hausdorff group unless otherwise is stated; by $e$ \ we  denote the unity of $G$. 

All topological spaces  are assumed to be Tychonoff (= completely regular and Hausdorff). 
The basic ideas and facts of the theory of $G$-spaces or topological
transformation groups can be found in Bredon \cite{bre:72} and in  Palais \cite{pal:60}.  
Our basic references on  proper group actions are   Palais~\cite{pal:61} and Abels   \cite{abe:78}.        
For the  equivariant theory of retracts the reader can see, for instance, \cite{ant:87},  \cite{ant:88} \cite{ant:99}, \cite{ant:jap} and \cite{ant:proper}. 

For the  convenience of the reader we recall, however,  some more  special definitions and facts.

\smallskip

Here we deal with $G$-spaces. If $X$ and $Y$ are two $G$-spaces then a continuous map $f:X\to Y$ is called a $G$-map, if $f(gx)=gf(x)$ for all $x\in X$ and $g\in G$. If a $G$-map is a homemorphism then it is called a $G$-homeomorphism.

If $X$ is a $G$-space and $H$ a subgroup of $G$ then, for a subset $S\subset X$, \ $H(S)$ denotes the $H$-saturation of $S$, i.e., $H(S)$= $\{hs | \ h\in H,\ s\in S\}$. In particular, $H(x)$  denotes the $H$-orbit $\{hx\in X | \ h\in H \}$ of $x$.  The quotient space  of all $H$-orbits is called the $H$-orbit space and denoted by $X/H$. 

If $H(S)$=$S$, then $S$ is said to be an $H$-invariant set. A  $G$-invariant set will simply be called an invariant set.  

For a closed subgroup $H \subset G$, by $G/H$ we will denote the $G$-space 
of cosets $\{gH | \ g\in G\}$ under the action induced by left translations.

If $X$ is a $G$-space and $H$  a closed normal subgroup of $G$, then the $H$-orbit space $X/H$  will always be regarded as a  $G/H$-space endowed with the following action of the group $G/H$:
$(gH)*H(x)=H(gx), \ \ \text{where} \ \ gH\in G/H, \ H(x)\in X/H$.

For any $x\in X$, the subgroup   $G_x =\{g\in G   \mid  gx=x\}$ is called  the stabilizer (or stationary subgroup) at $x$.

\smallskip

 Let  $X$ be  a $G$-space. Two subsets  $U$ and $V$ in  $X$  are called  thin relative to each other   \cite[Definition 1.1.1]{pal:61},  if the set 
$\langle U,V\rangle=\{g\in G | \ gU\cap V\ne \emptyset\}$
    has  a compact closure in $G$. 
   A subset $U$ of a $G$-space $X$ is called {\it small}, \ if  every point in $X$ has a neighborhood thin relative to $U$. A $G$-space $X$ 
is called  {\it  proper} (in the sense of R. Palais),  if   every point in  $X$ has a small neighborhood. We refer to the seminal paper of R. Palais \cite{pal:61} for further information about proper $G$-spaces.

\smallskip

In the present  paper we are especially interested in the class $G$-$\mathcal M$ of  all  metrizable proper $G$-spaces  that admit a compatible $G$-invariant metric. 
It is well-known that, for $G$ a compact group, the class $G$-$\mathcal M$
coincides with the class of {\it all} metrizable $G$-spaces (see \cite[Proposition 1.1.12]{pal:60}).
A  fundamental result of  R.~Palais \cite[Theorem~4.3.4]{pal:61}
states that if $G$ is a Lie group, then  $G$-$\mathcal M$ includes all {\it separable}, metrizable  proper $G$-spaces.

 Let us recall the definition of a twisted product $G/H\times_K S$, where $H$ is a closed normal subgroup of $G$,  $K$  any closed subgroup of $G$,  and   $S$  a $K$-space. 
 
$G/H\times_KS$ is the orbit space of the $K$-space $G/H\times S$, where $K$ acts on the Cartesian product $G/H\times S$ by $k(gH, s)=(gk^{-1}H, ks)$. Furthermore, there is a natural action of $G$ on $G/H\times_K S$ given by $g^\prime[gH, s]=[g^\prime gH, s]$, where $g'\in G$ and $[gH, s]$ denotes the  $K$-orbit of the point $(gH, s)$ in $G/H\times S$. The twisted products of the form $G\times_KS$ (i.e., when $H$ is  the trivial subgroup of $G$) are  of a particular interest in the theory of transformation groups (see \cite[Ch.~II, \S~2]{bre:72}). 

 \smallskip

A $G$-space  $Y$ is called an equivariant absolute neighborhood extensor  for the class  $G$-$\mathcal M$ 
(notation: $Y\in G$-{\rm ANE})  if,   for any  $X\in G$-$\mathcal M$  and any closed invariant subset $A\subset X$, every    $G$-map $f:A\to Y$ admits a $G$-map $\psi\colon U\to Y$ defined on 
  an invariant neighborhood $U$ of $A$ in $X$ such that $\psi|_A= f$. If, in addition, one  can always take $U=X$, then we say that $Y$ is an equivariant absolute extensor  for  $G$-$\mathcal M$ 
  (notation: $Y\in G$-AE). The map $\psi$ is called a $G$-extension of $f$.

\medskip

The following proposition was proved in \cite[Proposition 3.3]{ant:proper} and will be used in the proofs of Theorems \ref{T:almost} and \ref{T:00}.

\begin{proposition}[\cite{ant:proper}]\label{P:37} Let $H$ be a closed normal subgroup of $G$, $K$   a compact large subgroup of $G$,  and $S$ a $K$-space. If $S$ is a $K$-$ANE$, and all  $K\cap H$-orbits in $S$ are metrizable, then the twisted product $G/H\times_K S$ is a $G/H$-$ANE$.   
\end{proposition}

\smallskip

Let us recall  the well known definition of a slice \cite[p.~305]{pal:61}:

\begin{definition}\label{D:21} Let $X$ be a $G$-space and   $H$   a closed  subgroup of $G$. An $H$-invariant  subset $S\subset X$ is called an  $H$-slice in $X$, if $G(S)$ is open in $X$ and there exists  a $G$-map $f:G(S)\to G/H$ such that $S$=$f^{-1}(eH)$.  The saturation $G(S)$ is  called   a {\it tubular} set and $H$ is called a slicing group. 

  If  $G(S)=X$, then we say that $S$ is {\it a global} $H$-slice for $X$.   
\end{definition}
\medskip

The following  result of R. Palais \cite[Proposition 2.3.1]{pal:61} plays  a central role in the theory of topological transformation groups.

\begin{theorem}[Slice Theorem]\label{T:Sl} Let $G$ be a Lie group, $X$ be a proper $G$-space and $x\in X$. Then there exists a $G_x$-slice $S\subset X$ such that $x\in S$.
\end{theorem}

In our proofs we will also need  the following approximate version of the Slice Theorem  proved in  \cite[Theorem 3.6]{ant:jap} (see also \cite[Theorem 6.1]{ant:dikr})  which is valid for any locally compact group.

\begin{theorem}[Approximate Slice Theorem] \label{T:331} Let $G$ be any group, $X$  a proper $G$-space and  $x\in X$. Then for any  neighborhood $O$ of $x$ in $X$,   there exist  a   compact  large subgroup $K$ of $G$ with $G_x\subset K$, and a $K$-slice  $S$ such that  $x\in S\subset O$. 
\end{theorem}

Recall that here a subgroup $K\subset G$ is called  {\it large}, if the quotient space $G/K$ is locally connected and finite-dimensional (see \cite{ant:dikr}).

In the context of equivariant extension properties the  notion of a  large subgroup was first singled out in \cite{ant:94} (for compact groups)  and in \cite{ant:99} (for locally compact groups). Although  some geometric characterizations of this notion were available much earlier (see \cite[Section 3]{ant:dikr} and  the literature cited there), new characterizations through equivariant extension properties of the coset space $G/K$ were given in \cite[Proposition 6]{ant:99}, \cite[Proposition 3.2]{ant:jap} and \cite[Theorem 5.3]{ant:dikr}.

\medskip

The following result will be applied in the proofs of all three theorems below.


\begin{proposition}[{\cite[Proposition 3.4]{ant:jap}}]\label{P:large} Let $K$  be a compact  large  subgroup of $G$, and $X$  a  $G$-${\rm ANE}$ (respectively, a $G$-${\rm AE}$). Then  $X$ is  a $K$-${\rm ANE}$ (respectively, a $K$-${\rm AE}$).
\end{proposition}

\begin{remark}\label{R:K}  A careful analysis of  the proof of  \cite[Proposition 3.4]{ant:jap} shows that   this result is true
also for any compact subgroup $K$ of $G$ such that the coset space $G/K$ is just metrizable. Indeed, in the proof of 
 \cite[Proposition 3.4]{ant:jap} it is just needed that the twisted product $G\times_K S$ admits a $G$-invariant metric provided that $G/K$ and $S$ are metrizable. But this is true without assuming that $K$ is a large subgroup and this is proved explicitly in  \cite[Lemma 6.5]{ant:fund09}  and  \cite[Theorem 6.1]{ant:fund09}.
\end{remark}

\medskip
  
  The following proposition is well known  (see, e.g. \cite[Lemma 3.5]{abe:78}).
  
  \begin{proposition}\label{twist} Let  $H$ be a compact subgroup of $G$, $X$  a proper $G$-space and  $S$ a global $H$-slice of $X$.  Then  the  map  $\xi:G\times_H S\to X$ defined by $\xi([g, s])=gs$ is a $G$-homeomorphism.
\end{proposition}

\medskip

The following  equivariant version of   Hanner's open union theorem \cite[Theorem 19.2]{ha:52} is proved in \cite[Corollary 5.7]{ant:jap}. A  short and beautiful proof of Hanner's theorem 
was given by J. Dydak \cite[Corollary 1.5]{dy:pp}.

\begin{theorem}[\cite{ant:jap}]\label{T:Un} Let $Z\in G$-$\mathcal M$.
 If a  $G$-space $Y$ is the  union of a  family of  invariant open 
$G$-${\rm ANE}(Z)$ subsets $Y_\mu\subset Y$, $\mu\in \mathcal M$, then $Y$ is a $G$-${\rm ANE}(Z)$. 
\end{theorem}

   \medskip
  
  \section{ Proof  of Theorem \ref{T:0}}
  
   By Theorem~ \ref{T:331}, $X$ has  an open invariant cover by tubular sets of the form $G(S)$, \  where each $S$ is a $K$-slice with  the slicing group $K$  a   compact large  subgroup of $ G$. Then the orbit space $X/G$ is the  union of its open subsets of the form $G(S)/G$. According to  
Hanner's open union theorem  in \cite[Theorem 19.2]{ha:52} or \cite[Corollary 1.5]{dy:pp} (see also  
Theorem \ref{T:Un}), it suffices to show that each $G(S)/G$ is an {\rm ANE}.

To this end, we first observe that  each $G(S)$ is $G$-homeomorphic to the twisted product $G\times_{K} S$ (see Proposition \ref{twist}). This implies that $G(S)/G$ is homeomorphic to $(G\times_{K} S)/G$.
Since $X\in G$-{\rm ANE}, the tubular set  $G(S)$, being  an  open invariant subset of  $X$, is itself  a $G$-{\rm ANE}. Thus, $G\times_K S$ is a $G$-{\rm ANE}. Since the slicing group $K$  is a compact large subgroup of $G$, one can apply Proposition \ref{P:large}, according to which 
$G\times_{K} S$ is a $K$-{\rm ANE}. 
Each $K$-orbit in $X$ is contained in a $G$-orbit, and hence, is metrizable. Since $K$ is compact,  Theorem \ref{T:compactcase}  implies that
$(G\times_{K} S)/K$ is an {\rm ANE}. By Proposition \ref{retractdirect}, $(G\times_{K} S)/G$ is homeomorphic to a retract of $(G\times_{K} S)/K$, and hence, is itself  an {\rm ANE}. Consequently, $G(S)/G$ is an {\rm ANE},  as required.
\qed

   \medskip

\section{Proof of Theorem \ref{T:almost}}\label{almost}

Since $X\in G$-$\mathcal M$ the orbit space $X/G$ is metrizible, and hence, by   Abels \cite[Main Theorem]{abe:74}, $X$ admits a global $K$-slice
$S$ where $K$ is a maximal compact subgroup of $G$. Then, by Proposition \ref{twist}, $X$ is $G$-homeomorphic to the twisted product $G\times_KS$. 

 Observe that for every  maximal compact subgroup $K\subset G$, the coset space $G/K$ is metrizable. Moreover, 
 $G/K$ is homeomorphic to a Euclidean space (see  \cite[Corollary A6]{abe:74}).
 
    Therefore,  
   one can apply  Proposition \ref{P:large} and Remark \ref{R:K},    according to which $G\times_{K} S$ is a $K$-{\rm ANE}. 

Since $G$ is almost connected,    one can apply Proposition \ref{abels}, according to which $S$ is a $K$-equivariant retract of  $G\times_{K} S$, and hence, $S$ is  a $K$-{\rm ANE}.

Further, one has the following $G$-homeomorphism:
$$(G\times_K S)/H\cong G/H\times_ K S.$$
 Indeed, the map that sends the point $[g, s]_H$ of \ $(G\times_K S)/H$  to the point  $[gH, s]$ of \ $G/H\times_K S$ is a $G/H$-homeomorphism, where $[g, s]_H$ denotes the $H$-orbit of  $[g, s]$ in $G\times_KS$ (the easy verification is left to the reader). 

Next we  observe that every $K\cap H$-orbit in $S$ is metrizable since it is  contained in the corresponding $H$-orbit  in $X$, which is metrizable by the hypothesis. Further, since $S\in K$-{\rm ANE},  it then follows from Proposition \ref{P:37}  that the twisted product $G/H\times_ K S$ is a $G/H$-{\rm ANE}. This yields that  $(G\times_ K S)/H\in G/H$-{\rm ANE}, and since $X/H$ is $G/H$-homeomorphic to 
$(G\times_ K S)/H$, we conclude that $X/H\in G/H$-{\rm ANE}, as required.
 \qed

  \medskip
  
  \section{ Proof  of Theorem \ref{T:00}}
    
 By Theorem~ \ref{T:Sl},
 $X$ has  an open invariant cover by tubular sets of the form $G(S)$, \  where each $S$ is a $K$-slice  with the slicing group  $K$ a compact   subgroup of $G$.  Then the $G/H$-space $X/H$ is the  union of its open $G/H$-invariant subsets of the form $G(S)/H$. According to  Theorem \ref{T:Un}, it suffices to show that each $G(S)/H$ is a $G/H$-{\rm ANE}.

To this end, we first observe that  each $G(S)$ is $G$-homeomorphic to the twisted product $G\times_{K} S$ (see Proposition \ref{twist}). 

This yields that $G(S)/H$ is $G/H$-homeomorphic to 
$(G\times_ K S)/H$.
Since $X\in G$-{\rm ANE}, the tubular set  $G(S)$, being  an  open invariant subset of  $X$, is itself  a $G$-{\rm ANE}. Thus, $G\times_{K} S$ is a $G$-{\rm ANE}. Since $G$ is a Lie group, we infer that $G/K$ is   metrizable (moreover,
evidently, $K$ is a compact large subgroup in this case). Then one can apply Proposition \ref{P:large}, according to which 
$G\times_{K} S$ is a $K$-{\rm ANE}. 
By   Proposition \ref{P:1}, $S$ is a $K$-equivariant neighborhood retract of  $G\times_{K} S$, and hence,  $S$ is  a $K$-{\rm ANE}.

Further, as we mentioned in the proof of Theorem \ref{T:almost}, one has the  following $G$-homeomorphism: 
$$(G\times_K S)/H\cong G/H\times_ K S.$$
 Since $S\in K$-{\rm ANE},  it then follows from Proposition \ref{P:37}  that the twisted product $G/H\times_ K S$ is a $G/H$-{\rm ANE}. This yields that  $(G\times_ K S)/H\in G/H$-{\rm ANE}, and since, $G(S)/H$ is $G/H$-homeomorphic to 
$(G\times_ K S)/H$, we conclude that $G(S)/H\in G/H$-{\rm ANE}, as required.
 \qed

\medskip

\section{Lifting  of equivariant embeddings and extension properties of orbit spaces}\label{lifting}

The lifting properties of $G$-equivariant closed embeddings for a compact  acting group $G$ were first established in \cite{ant:preserv}.  Below, in  Theorems \ref{T:lifting1} and \ref{T:anygroup} 
we generalize these results to the case of proper actions of non-compact  groups.  In turn, this allows us  to strengthen  
 Theorems~\ref{T:0}, \ref{T:almost} and \ref{T:00},  discarding  in their statements the hypothesis about the properness of the $G$-space $X$.

\begin{theorem}\label{T:lifting1} 

 Let   $G$ be either a Lie group or an almost connected  group, and let $H$ be  a closed  normal subgroup  of $G$. Suppose that    $A\in G$-$\mathcal M$ and
 $f:A/H\hookrightarrow  B$
 is a $G/H$-equivariant closed embedding into  a $G/H$-space  $B\in G/H$-$\mathcal M$. Then there exist
 a $G$-space  $Z\in G$-$\mathcal M$ and a $G$-equivariant closed embedding
 $\phi:A\hookrightarrow Z$
such that $Z/H$ is a  $G/H$-invariant neighborhood of $A/H$ in $B$
and $q\circ\phi=f\circ p$, where \ $p:A\to A/H$ and \  $q:Z\to
Z/H$ are the $H$-orbit maps.
\end{theorem}

\begin{proof} According to \cite[Theorem~6.1]{ant:elena}, it can be assumed that A is a
closed $G$-invariant subset of a $G$-AE space $L\in G$-$\mathcal M$. Then $A/H$ is  a closed invariant
subset of the $G/H$-space $L/H$.

\begin{displaymath}
 \xymatrix{ Z  \ar[d]^{q} & A \ar @{_{(}->} [l]_{\phi} \ar @{^{(}->} [r]^{} \ar[d]^{p}
& L  \ar[d]^{r}  \\
U  \ar @/_/[rr]_F  & A/H \ar @{_{(}->} [l]_{f} \ar @{^{(}->} [r]^{}  & L/H }
\end{displaymath}

\medskip

\noindent 
Now, by Theorem~\ref{T:almost} (for almost connected groups) and Theorem
\ref{T:00} (for Lie groups), $L/H\in G/H$-ANE. Therefore, there exist a $G/H$-equivariant
extension. $F:U\to L/H$ of the $G/H$-map
$f^{-1}:f(A/H)\to A/H \hookrightarrow L/H$ defined on some $G/H$-neighborhood $U$ of the set $A/H$ in $B$. 

Let $r:L \to L/H$ be the $H$-orbit projection. Denote by
$Z$ the pull-back (or fiber product) of  $L$  with respect to the maps $F$ and $r$, i.e.,
$$ Z=\{(u,x)\in U\times L\ | \  F(u)=r(x)\}.$$
 We will consider the coordinate-wise defined action of the group $G$ on $Z$, i.e., $g(u,x)=(gu,gx)$ for
$g\in G$ and $(u, x)\in Z$. Let $h:Z/H\to U$ be the map defined by the formula $h(q(u,x))=u$, where
$q:Z\to Z/H$ is the $H$-orbit projection and $(u,x)\in  Z$. It is clear that $h$ is a well-defined
$G/H$-equivariant map. It can easily be shown (and this is well known, see \cite[Ch.4,
Proposition~4.1]{huse:94}) that $h$ is a homeomorphism.

 On the other hand the product $U\times
L$ is a proper $G$-space because $L$ is so. Besides, since $U$ and $L$ admit $G$-invariant metrics,
we infer that $U\times L$ also has a $G$-invariant metric. Thus $U\times L\in G$-$\mathcal M$, which
implies that $Z\in G$-$\mathcal M$.

It remains to define the $G$-equivariant embedding $\phi:A \hookrightarrow Z$ by the formula
$\phi (a)=(f(p(a)),a)$ for $a\in$ A. This completes the proof.
\end{proof}

\medskip

\begin{theorem}\label{T:anygroup}
 Let   $G$ be any locally compact   group. Assume  that    $A\in G$-$\mathcal M$ and
 $f:A/G\hookrightarrow  B$ is a closed embedding into  a metrizable  space $B$. Then there exist
 a $G$-space  $Z\in G$-$\mathcal M$ and a $G$-equivariant closed embedding
 $\phi:A\hookrightarrow Z$
such that $Z/G$ is a  neighborhood of $A/G$ in $B$
and $q\circ\phi=f\circ p$, where \ $p:A\to A/G$ and \  $q:Z\to
Z/G$ are the $G$-orbit maps.
\end{theorem}
\begin{proof} Repeat the proof of Theorem \ref{T:lifting1} with $H=G$, where you simply  need to replace the reference to Theorem~\ref{T:00} with a reference to Theorem~\ref{T:0}.
\end{proof}

\medskip

By virtue of Theorems \ref{T:lifting1} and \ref{T:anygroup}, one can dropp the hypothesis about the properness of the $G$-space $X$ in Theorems~\ref{T:0},  \ref{T:almost} and \ref{T:00}.

\begin{theorem}[Non-proper actions of Lie groups and almost connected groups]\label{T:NP1}
 Let   $G$ be either a Lie group or an almost connected  group,
 and let $X$ be any  $G$-${\rm {\rm ANE}}$. Assume that $H$ is a closed  normal subgroup of $G$
such that all  $H$-orbits in $X$ are metrizable. Then the $H$-orbit space $X/H$ is a $G/H$-${\rm
ANE}$.
\end{theorem}

\begin{proof} Let  $B\in G/H$-$\mathcal M$. Let $L$ be  a closed
$G/H$-invariant subset of $B$ and let $s:L\to X/H$ be  a $G/H$-map. Define $A\subset L\times X$ to
be the pull-back of the $G$-space $X$ with respect to $s$ and $t$, where $t:X\to X/H$ is the
$H$-orbit map. Then $A$ is a $G$-invariant subspace of $L\times X$ endowed with the diagonal action
of $G$, and we have $A/H=L$ (see \cite[Ch.~4, Proposition~4.1]{huse:94}). Since the $H$-orbit of
each point $a=(l,x)\in A$ lies in the metrizable  space $L\times H(x)$, we conclude that $H(a)$ is
metrizable too. So, all $H$-orbits of the $G$-space $A$, as well as its $H$-orbit space $A/H=L$, are
metrizable. By \cite[Theorem 6.1]{ant:fund09}, $A$ is metrizable. Now applying  Theorem~\ref{T:lifting1},
 we get a $G$-space $Z\in G$-$\mathcal M$  with $Z/H$ a
$G/H$-invariant neighborhood of $L$ in $B$ such that $A$ is a closed
$G$-invariant subspace of $Z$.

Let $\psi:A\to X$ be the restriction of the projection $L\times
X\to X$. Since $X\in G$-ANE, there
exist a $G$-invariant neighborhood $U$ of $A$ in $Z$  and a $G$-extension $\alpha:U\to X$ of the $G$-map $\psi$.
It is easy to see that the induced $G/H$-map $\beta:U/H\to X/H$ is
the desired $G/H$-extension of $s$. This completes the proof.
\end{proof}

\medskip

\begin{theorem}[Non-proper actions of locally compact groups with $H=G$]\label{T:NP2}
Let $G$ be a locally compact  group and $X$   any $G$-space such that  all  $G$-orbits in $X$ are metrizable. 
 If  $X$ is a $G$-${\rm ANE}$ then the $G$-orbit space  $X/G$ is an {\rm ANE}. 
\end{theorem}
\begin{proof} Repeat the proof of Theorem \ref{T:NP1}, where you  simply need to replace the reference to Theorem~\ref{T:lifting1} with a reference to Theorem~\ref{T:anygroup}. 
\end{proof}

\medskip

\section{Appendix}\label{append}

In this section we will present in detail some simple and very useful propositions that culminate in a proof of Proposition \ref{retractdirect}.

 \begin{proposition}\label{closed}
 
Let $G$ be a topological group, $H$ a closed subgroup of $G$, and $X$ a $G$-space.
 Then for any closed subset $B\subset X$, the set $A=\{(h^{-1}, hb) \mid h\in H, \  b\in B\}$ is  closed in the product
$G\times X$.
 \end{proposition}

\begin{proof} 
Assume that  $(g, x)$ is a closure point of $A$ and prove that $(g, x)\in A$.   
 There exist nets  $(h_i)\subset H$ and  $(b_i)\subset B$   such that the net $(h_i^{-1}, h_ib_i)$ converges to $(g, x)$. 
This yields that 
$(h_i^{-1})$ converges to $g$ and  $(h_ib_i)$ converges to $x$. Since $H$ is closed we infer that $g\in H$. Clearly, $b_i=h_i^{-1}(h_ib_i)$ converges to $gx$. Since $b_i\in B$ and $B$ is closed, we infer that $gx\in B$. Thus, $(g, x)=(g, g^{-1}gx)$ where   $g\in H$ and $gx\in B$. This shows  that 
 $(g, x)\in A$, as required.

\end{proof}

The following proposition is well-known in the literature only for the compact subgroup $H$.

 \begin{proposition}\label{closed}
 
Let $G$ be a topological group and $H$ a closed subgroup of $G$. If $X$ is an $H$-space, then  the map $\iota:X \hookrightarrow G \times_H X, \; \iota(x) = [e,x]$ is a closed $H$-embedding, where $e\in G$ is the unit element.
 \end{proposition}
 
 \begin{proof}
  Since  $\iota$ is the composition $X \stackrel{j}{\rightarrow} G \times X
 \stackrel{p}{\rightarrow} G \times_H X$, where $j$ is the closed embedding $j(x)=(e,x)$ and $p$ is 
 the $H$-orbit map, we infer that $\iota$ is continuous.
 
 Let $A$ be a closed subset of $X$. To prove that $\iota (A)$ is closed in $G \times_H X$ it suffices to prove that the inverse image $p^{-1}\big
 (\iota(A)\big)$ is closed in $G\times X$, where $p:G\times X\to G\times_HX$ is the $H$-orbit map. We have that
 $$p^{-1}\big(\iota(A)\big)=\{(g, x)\in G\times X \mid [g, x]=[e, a] \ \text{for some} \ a\in A\}.$$
 But the equality $[g, x]=[e, a]$ means that $(g, x)=(h^{-1}, ha)$ for some $h\in H$. Consequently, 
  $$p^{-1}\big(\iota(A)\big)=\{(h^{-1}, ha)\in G\times X  \mid h\in H, \ a\in A\},$$
  which, by Proposition \ref{closed}, is closed in $G\times X$. Thus, $\iota$ is a closed map.
  
  Further, the map  $\iota$ is injective since
 \[ [e,x]=[e,y] \Longleftrightarrow (e,y)=(eh^{-1},hx) \quad \text{for some} \; h \in H \]
 and, in this case, $h=e$, so $y=x$. Hence, $\iota$ is a closed embedding.

 If $x \in X$ and $h \in H$, then 
 \[ \iota(hx) = [e, hx] = [h, x] = h[e,x] = h\iota(x)   \]
 showing that  $\iota$ is $H$-equivariant, as required.  
 \end{proof}

 \begin{proposition}\label{homeo}
 Let $G$ be a topological group and $H$ any subgroup of  $G$. If $X$ is an $H$-space, then 
 the $G$-orbit space $({G \times_H X})/G$ is homeomorphic to the $H$-orbit space $X/H$.
 \end{proposition}
 \begin{proof}
 The projection  $\pi: G \times X \to X$ is a continuous open map which induces a continuous open map 
  $\alpha:G \times_H X \to X/H, \; [g,x] \mapsto H(x)$ between $H$-orbit spaces.
               
Observe  that $\alpha$ is constant on the $G$-orbits of the $G$-space   $G \times_H X$. Indeed 
for every  $g' \in G$ one has  $\alpha(g'[g,x])=\alpha([g'\cdot g,x])=H(x)=\alpha([g,x])$, as required. Then $\alpha$ induces a continuous  bijective map
 $r:({G \times_H X})/G \to X/H$ which makes to commute the following diagram:

              \[  \xymatrix{
              G \times_H X \ar[r]^\alpha \ar[d] \ar[d]_q & X/H, \\
              \frac{G \times_H X}{G}  \ar@{-->}[ur]^r  &
              }   \] 
    where $q$ is the $G$-orbit map that sends $[g, x]$ to its orbit $G([g, x])$.     Since $\alpha$ is open, it follows from the commutativity of the diagram that  the induced map $r$ also is open, and hence, it is the desired homeomorphism.     
              
    \end{proof}

 \begin{proposition}\label{retract}
 Let $G$ be a topological group and $H$ a closed subgroup of  $G$. If $X$ is an $H$-space, then 
 the $H$-orbit space $X/H$ is homeomorphic to a retract of the $H$-orbit space $({G \times_H X})/H$.
  \end{proposition}
 \begin{proof}By Proposition \ref{closed}, the map $\iota: X\hookrightarrow G\times_HX$, $x\mapsto [e, x]$ is a closed
 $H$-embbeding. This induces a closed embedding 
 $\tilde \iota: X/H\hookrightarrow (G\times_HX)/H$. Consequently, it suffices to prove that the image $\mathcal Im\,  \tilde \iota $ is a retract of $(G\times_HX)/H$.

  For every $[g, x]\in G \times_H X$ we will denote by $[g, x]_H$ the $H$-orbit in the $G$-space $G \times_H X$. 
  Clearly,  $\mathcal Im\,  \tilde \iota =\{[e, x]_H \mid x\in X\}.$

 Define a map $r: (G\times_HX)/H\to \mathcal Im\,  \tilde \iota $ \  by the rule:  $r: [g, x]_H\mapsto [e, x]_H$.  This map is well defined since for any 
 $h\in H$ one has 
 $$r: [gh^{-1}, hx]_H\mapsto [e, hx]_H=[h, x]_H= (h[e, x])_H= [e, x]_H,$$
 as required. 
 
The projection $G\times X\to X$  is $H$-equivariant and thus induces a continuous map $(G\times_HX)/H\to X/H$. 
This is given by $[g, x]\mapsto H(x)$,  and hence, factors as a continuous map $f:  \frac{G \times_H X}{H} \longrightarrow     X/H$, \ $[g, x]_H\mapsto H(x)$.
  
 The continuity of $r$ follows from the fact that it is the composition of the following two continuous maps:
   $$ \frac{G \times_H X}{H} \stackrel{ f}{\longrightarrow}     X/H  \stackrel{  \tilde \iota}{\longrightarrow}     \mathcal Im\,  \tilde \iota, $$ 
   
$$[g, x]_H \mapsto H(x) \mapsto [e, x]_H.$$    
    
    \medskip
        
 Besides, if $[e, x]_H \in \mathcal Im\,  \tilde \iota$,  then $r\big([e, x]_H\big)=  [e, x]_H$, so $r$ is the desired retraction.
  Thus, $X/H$ is homeomorphic to  $\mathcal Im\,  \tilde \iota$ \ which is  a retract of the $H$-orbit space $({G \times_H X})/H$, as required.             
  \end{proof}
    
    \medskip
    
    Now, as a simple combination of Propositions  \ref{homeo} and \ref{retract}, we get  Proposition \ref{retractdirect} already stated in Section \ref{intro}.

\medskip

  We conclude the paper with the following conjecture.
  
  \begin{conjecture}Let $G$ be a locally compact  group, $K$ a compact large  subgroup of $G$, and $S$  a $K$-space. Then $S$ is a neighborhood $K$-equivariant retract of the twisted product $G\times_K S$.  
    \end{conjecture}
    
    This conjecture first appeared in \cite[Question 4.4]{ant:proper} in the  form of a question. 
    Note that this is true for any Lie group $G$ (see Proposition \ref{P:1})  and for any  almost connected group $G$ (see Proposition \ref{abels}).
    The validity of this conjecture will allow us to extend the proof of Theorem \ref{T:00} to the case of proper actions of arbitrary locally compact groups.

\bibliographystyle{amsplain}

\bibliography{triquot}
\end{document}